\documentclass[12pt,reqno]{amsart}

\usepackage{verbatim,amssymb,amsmath,amscd,latexsym,amsbsy,m athrsfs}
\usepackage{enumerate}
 \input{xy}
\xyoption{all}

\newtheorem{thm}{Theorem}[section]
\newtheorem{pro}[thm]{Proposition}
\newtheorem{cor}[thm]{Corollary}
\numberwithin{equation}{section}

\theoremstyle{definition}
\newtheorem{example}[thm]{Example}
\newtheorem{df}[thm]{Definition}
\newtheorem{remark}[thm]{Remark}

\newcommand{\Hom}{{\rm Hom}}
\newcommand{\Vect}{{\rm Vect}}
\newcommand{\Rep}{{\rm Rep}}
\newcommand{\BL}{{\rm BL}}

\newcommand{\PP}{\mathbb{P}}
\newcommand{\cO}{\mathcal{O}} 
\newcommand{\scrX}{\mathcal{X}}
\newcommand{\scrY}{\mathcal{Y}}
\newcommand{\scrV}{\mathcal{V}}
\newcommand{\scrF}{\mathcal{F}}
\newcommand{\scrI}{\mathcal{I}}
\newcommand{\cK}{\mathcal{K}}

\newcommand{\Z}{\mathbb{Z}}

\newcommand{\ilim}{\mathop{\varprojlim}\limits}

\begin{document}

\title[Higher dimensional formal orbifolds and orbifold bundles]{Higher dimensional formal 
orbifolds and orbifold bundles in positive characteristic}

\author[I. Biswas]{Indranil Biswas}

\address{School of Mathematics, Tata Institute of Fundamental
Research, Homi Bhabha Road, Bombay 400005, India}

\email{indranil@math.tifr.res.in}

\author[M. Kumar]{Manish Kumar}

\address{Statistics and Mathematics Unit, Indian Statistical Institute,
Bangalore 560059, India}

\email{manish@isibang.ac.in}

\author[A.J. Parameswaran]{A. J. Parameswaran}

\address{School of Mathematics, Tata Institute of Fundamental
Research, Homi Bhabha Road, Bombay 400005, India}

\email{param@math.tifr.res.in}



\begin{abstract}
In \cite{KP}, the last two authors introduced formal orbifold curves defined over an
algebraically closed field of positive characteristics. They studied both \'etale and Nori
fundamental group schemes associated to such objects. Our aim here is to study the
higher dimensional analog of these objects objects and their fundamental groups.
\end{abstract}

\maketitle

\section{Introduction}

Given a quasiprojective variety $X$ defined over an algebraically closed field $k$ of positive characteristic,
and a base point $x_0\, \in\, X$, the Nori fundamental group $\pi^N(X,\, x_0)$ is defined using the torsors
on $X$ for finite $k$--group schemes. This construction gives the \'etale fundamental group of $X$ if we
restrict to the reduced group schemes. When $X$ is complete, $\pi^N(X,\, x_0)$ has a Tannakian description
using the essentially finite bundles on $X$ introduced in \cite{No0}. The homomorphisms
between the fundamental groups induced by \'etale morphisms of varieties are well understood. The paper
\cite{KP}, which is a predecessor of the present work, originated from attempts to understand the
homomorphisms between the fundamental groups induced by ramified maps between curves.

We quickly recall the aspects of \cite{KP} that connect it to the present work. Given a finite morphism 
$f\,:\,X\,\longrightarrow\, Y$ between curves, consider all finite morphisms $g\,:\,Z\,\longrightarrow\, Y$ that 
are locally dominated by $f$. This will form an inverse system, and by taking corresponding Galois extensions, 
it is possible define a group obtained by the inverse limits. This is made precise by introducing a branch data 
on $Y$ and a condition in terms of this branch data is imposed on these coverings of $Y$. Those branch data 
coming from a global finite map are referred to as geometrical branch data. In \cite{KP}, a class of bundles on 
those coverings are defined, and it is shown there that they form a tensor abelian category; the Tannakian 
dual of this tensor abelian category is called orbifold fundamental group with respect the orbifold structure 
defined by the branch data.

Here we consider the questions addressed in \cite{KP} in the set-up of higher dimensional 
varieties. We recall the definition of formal orbifolds $(X,\, P)$, where $X$ is a normal proper variety defined 
over $k$, and $P$ is a branch data on $X$ (see Section \ref{sec2}). Associated to $(X,\, P)$
is a tensor abelian category ${\rm Vect}^f(X,P)$ (see \eqref{vx}). It is defined by taking
equivariant essentially finite bundles on suitable ramified Galois coverings $Y$ of $X$ whose
ramifications are are controlled by $P$.

After fixing a base point $x\, \in\, X$ outside $P$, this tensor abelian category
produces a proalgebraic group scheme which is denoted by $\pi^N((X,P),\,x)$. Our first theorem
is the following (see Theorem \ref{ses-thm}):

\begin{thm}\label{in-ses-thm}
Let $f\,:\,(Y,\,O)\,\longrightarrow\,(X,\,P)$ be an \'etale $\Gamma$--Galois cover of 
projective formal orbifolds.
Take a point $y\, \in\, Y$. There is a natural exact sequence
$$1\,\longrightarrow\,\pi^N(Y,\,y)\,\stackrel{i}{\longrightarrow}\, \pi^N((X,P),\, f(y))\,
\stackrel{q}{\longrightarrow}\, \Gamma\,\longrightarrow\, 1\, .$$
\end{thm}

Let $X^o$ be an open dense subset of a normal projective variety, we define (\ref{df:new_nori}) a fundamental 
group scheme $\pi^n(X^o,\,x)$ as inverse limit of $\pi^N(X,\,P)$ where the limit is taken over branch data $P$ whose 
branch locus is disjoint from $X^o$. We observe that $\pi^n(X^o,\,x)$ is a quotient of the Nori fundamental group 
$\pi^N(X^o,\,x)$ (Proposition \ref{prohs}). We also show that $\pi^n(X^o,\,x)$ classifies finite group scheme torsors 
over $X^o$ which \'etale locally extends to $X$ (Proposition \ref{pro-characterize-pi^n}).

Our second theorem (Theorem \ref{thm2}) identified the kernel of the natural homomorphism
$$
\pi^N((X,P),\, x)\, \longrightarrow\, \pi_1^{et}((X,P),\,x))\,.$$

\section{Formal orbifold and Orbifold bundles}\label{sec2}

Let $X$ be a normal variety $X$ defined over a perfect field $k$.
We recall from \cite[Section 3]{formal.orbifolds} the definition of a branch data on $X$.
Let $x\,\in \,X$ be a point of codimension at least one, and let $U$ be 
an affine open connected neighborhood of $x$; we note that $U$ is integral
because $X$ is normal. 
Again normality of $X$ implies that the completion $\widehat{\cO_X(U)}^x$ of the coordinate 
ring $\cO_X(U)$ along $x$ is an integral domain. Let $\cK^x_X(U)$ denote the field of fractions 
for $\widehat{\cO_X(U)}^x$. The fraction field of $\widehat \cO_{X,x}$ will be denoted by
$\cK_{X,x}$.

A \emph{quasi branch data} $P$ on $X$ assigns to every such pair $(x,\, U)$
a finite Galois extension of $\cK^x_X(U)$ in a fixed algebraic closure of $\cK^x_X(U)$, which is denoted by $P(x,U)$, such that that
the following compatibility conditions hold:
\begin{enumerate}
\item $P(x_1,\,U)\,=\,P(x_2,\,U)\cK_X^{x_1}(U)$, where $x_1\,\in\, \overline{\{x_2\}}$,
and $U$ is an affine open connected neighborhood of $x_1$ and $x_2$.

\item For $x\,\in\, V\,\subset\, U\,\subset\, X$, with $U$ and $V$ affine open connected subsets, we have
$P(x,V)\,=\,P(x,U)\cK_X^x(V)$.
\end{enumerate}

Define $P(x)\,:=\,P(x,U)\cK_{X,x}$; note that $P(x)$ is independent of the choice of $U$. Also, define
$$\BL(P)\,:=\,\{x\,\in\, X\,\mid\,\widehat\cO_{X,x} ~\text{ is branched in }~ \cK_{X,x} \}\, .$$
A quasi branch data $P$ is called a branch data if $\BL(P)$ is a closed subset
of $X$ of codimension at least one. This $\BL(P)$ is called the branch locus of $P$.

Note that if $\dim X\,=\,1$, then $P(x,\,U)\,=\,P(x)$ (i.e., it is independent of $U$), and hence it
agrees with the notion in \cite{KP}.

The branch data in which all the Galois extensions are trivial is called the trivial branch data, and it is denoted 
by $O$. For a finite morphism $f\,:\,Y\,\longrightarrow\, X$ of normal varieties, the natural 
branch data associated to $f$ will be denoted by $B_f$.

We recall the definition of formal orbifolds from \cite{formal.orbifolds}. As before,
$k$ is a perfect field. A \emph{formal orbifold} 
over $k$ is a pair $(X,\,P)$, where $X$ is a normal finite type scheme over $k$ and $P$ is a branch data on $X$.

A morphism of formal orbifolds $f\,:\,(Y,\,Q)\,\longrightarrow\, (X,\,P)$ is a quasi-finite dominant separable 
morphism $f\,:\, Y\,\longrightarrow\,X$ such that for all points $y\,\in\, Y$ of codimension at least one and some 
affine open neighborhood $U$ of $f(y)$, we have $$Q(y,f^{-1}(U))\,\supset\, P(f(y),U)\, .$$
It is said to be \'etale 
if the extension $Q(y)/P(f(y))$ is unramified, for all $y\,\in\, Y$ of codimension at least one. Moreover, $f$
is called a covering morphism (or simply a covering) if it is also proper.

A formal orbifold $(X,\,P)$ is called \emph{geometric} if there exist an \'etale cover $(Y,\,O)\,
\longrightarrow\, (X,\,P)$ and in this case $P$ is called a geometric branch data \cite{KP}.

Let $(Y,\,O)\,\longrightarrow\, (X,\,P)$ be an \'etale $\Gamma$--Galois covering of formal orbifolds. Like in \cite{KP}, we define 
vector bundles on $(X,\,P)$ as the $\Gamma$--equivariant bundles on $Y$, while morphisms between two
vector bundles on $(X,\,P)$
are defined to be the $\Gamma$--equivariant homomorphisms between the corresponding $\Gamma$--bundles on $Y$. For
the case of curves,
it was shown in \cite{KP} that this definition does not depend on the 
choice of the \'etale cover. The key point is that if $(Y_i,\,O)\,\longrightarrow\, (X,\,P)$ are
\'etale $\Gamma_i$--covers for $i\,=\,1,\,2$, then take
an \'etale $\Gamma$--cover $(Y,\,O)\,\longrightarrow\, (X,\,P)$ that dominates these two covers
(for instance $Y$ can be the 
normalized fiber product of $Y_1$ and $Y_2$). It follows that $Y\,\longrightarrow\,Y_i$ are Galois \'etale covers,
and then using Galois descent it is shown that the pullback functor defines an equivalence of
category of $\Gamma$--bundles of $Y$ and the category of $\Gamma_i$--bundles on $Y_i$. (See \cite[Lemma 3.3, Lemma 3.4, 
Proposition 3.6]{KP} for the proof.) It should be clarified that the proofs of these results in \cite{KP} do not 
use the hypothesis in \cite{KP} that $Y$ is a curve.

Now assume the base field $k$ to be algebraically closed. Let $X$ be a smooth proper variety 
over $k$. A vector bundle on $(X,\,P)$ is called stable (respectively, semi-stable)
if the corresponding $\Gamma$--equivariant bundle on $Y$ is equivariantly 
stable (respectively, equivariantly semistable).
However, an equivariant vector bundle is equivariantly 
semistable if and only if the underlying vector bundle is semistable.
Similarly, a vector 
bundle on $(X,\,P)$ is called essentially finite if the corresponding $\Gamma$--equivariant 
bundle on $Y$ is essentially finite. However, an equivariant vector bundle is equivariantly 

The tensor product and duals of vector bundles on $(X,\,P)$ are defined in the usual way. This 
makes the category
\begin{equation}\label{vx}
{\rm Vect}^f(X,P)
\end{equation}
of essentially finite bundles a Tannakian category, and any closed point $x\,\in\, X$ outside 
support of $P$ defines a fiber functor from $\Vect^f(X,P)$ to the category of $k$--vector 
spaces. Hence we define $\pi^N((X,P),\,x)$ to be the automorphism of this fiber functor. Note 
that if $P$ is the trivial branch data, then $\pi^N((X,P),\,x)$ is the
fundamental group $\pi^N(X,\,x)$ corresponding to the essentially finite bundles
\cite{No0}, \cite{No} (its definition is recalled in Section \ref{se3}.

\section{Basic properties of $\pi^N(X,P)$}\label{se3}

Let $f\,:\,(Y,\,O)\,\longrightarrow\,(X,\,P)$ be an \'etale $\Gamma$--Galois cover of 
projective formal orbifolds.

\begin{thm}\label{ses-thm}
Take a point $y\, \in\, Y$. There is a natural exact sequence
$$1\,\longrightarrow\,\pi^N(Y,\,y)\,\stackrel{i}{\longrightarrow}\, \pi^N((X,P),\, f(y))\,
\stackrel{q}{\longrightarrow}\, \Gamma\,\longrightarrow\, 1\, .$$
\end{thm}

\begin{proof}
Let $E$ be a $\pi^N((X,P), \, y)$--module, meaning it is an essentially finite vector bundle on $(X,\, P)$.
So $E$ is also an essentially finite vector bundle on $Y$. Hence we have a
homomorphism
\begin{equation}\label{ei}
i\, :\, \pi^N(Y)\,\longrightarrow\, \pi^N(X,P)
\end{equation}
(the base point is suppressed).
We note that any essentially finite vector bundle $F$ on $Y$ is a sub-bundle of the 
$\Gamma$--equivariant bundle $\bigoplus_{\gamma\in \Gamma} \gamma^*F$ on $Y$; this direct sum 
$\bigoplus_{\gamma\in \Gamma} \gamma^*F$ is essentially finite because $F$ is so. Consequently, 
$\bigoplus_{\gamma\in \Gamma} \gamma^*F$ is an essentially finite vector bundle on $(X,\, P)$. 
Hence the homomorphism $i$ in \eqref{ei} is a closed immersion \cite[p.~139, Proposition 
2.21(b)]{DM}.

Given a $\Gamma$--module $V$, we have the $\Gamma$--equivariant vector bundle $$Y(V)\ :=\,
Y\times V\, \longrightarrow\, Y\, ;$$
here $\Gamma$ acts diagonally on $Y\times V$ using its actions on $Y$ and $V$. Since $Y(V)$ is
essentially finite, it defines an essentially finite bundle on $(X,\, P)$. This construction produces a homomorphism
$$
q\, :\, \pi^N((X,P))\,\longrightarrow\, \Gamma
$$
(the base point is suppressed).
This $q$ is surjective because the above functor from $\Gamma$--modules to $\Vect^f(X,P)$ (defined in \eqref{vx})
is fully faithful \cite[p.~139, Proposition 2.21(a)]{DM}.

The composition $q\circ i$ is evidently trivial, because the vector bundle underlying the
$\Gamma$--equivariant bundle $Y(V)$ is trivial.

The inclusion homomorphism $\text{kernel}(q)\, \hookrightarrow\, \pi^N((X,P))$ corresponds to the 
forgetful functor that simply forgets the $\Gamma$--action on a $\Gamma$--equivariant vector 
bundle on $Y$. From this it follows that $\text{kernel}(q)\,=\, \text{image}(i)$. This 
completes the proof.
\end{proof}

Let $k$ be an algebraically closed field of positive characteristic. Take a reduced and 
connected $k$--scheme $X$, and fix a rational point $x\, \in\, X$. We recall from \cite{No} the 
construction of a profinite group-scheme over $k$ associated to the pair $(X,\, x)$. Consider 
all quadruples of the form $(G,\, Y,\, f,\, y)$, where
\begin{itemize}
\item $G$ is a finite group-scheme defined over $k$,

\item $f\, :\, Y\, \longrightarrow\, X$ is a $G$--torsor, and

\item $y\, \in\, Y$ is a rational point such that $f(y) \,=\, x$.
\end{itemize}
A morphism $(G,\, Y,\, f,\, y)\, \longrightarrow\, (G',\, Y',\, f',\, y')$
between two such quadruples is a pair of the form $(\rho,\, \varphi)$, where
$\rho\, :\, G\,\longrightarrow\, G'$ is a homomorphism of $k$--group-schemes
and $\varphi\, :\, Y\, \longrightarrow\, Y'$ is a morphism, such that
\begin{itemize}
\item $f'\circ\varphi\,=\, f$,

\item $\varphi(y)\,=\, y'$,

\item the morphism $\varphi$ is $G$--equivariant, for the action of $G$ on the
$G'$--torsor $Y'$ given by $\rho$.
\end{itemize}
Let $N(X,x)$ denote the category constructed using these quadruples and morphisms between
them.

The category $N(X,x)$ forms an inverse system. Nori proved that the inverse limit
$$
 \underset{{\sf N}(X,x)}{\ilim}G
$$
exists as a profinite group-scheme over $k$ \cite[Chapter 2, Proposition 2]{No}. This inverse
limit will be denoted by $\pi^N(X,\, x)$. When $X$ is a projective variety, this profinite
group-scheme $\pi^N(X,\, x)$ coincides with the Tannaka dual of the category of essentially
finite vector bundles on $X$ \cite[Chapter 1, Proposition 3.11]{No}.

\begin{df}\label{df:torsor-extends}
Let $X^o$ be an open dense subset of a normal projective variety $X$. Let $G$ be a finite group
scheme, and let $$Z^o\,\longrightarrow\, X^o$$ be a $G$--torsor.
We say that this $G$--torsor $Z^o$ \'etale locally extends to $X$ if there exist a connected \'etale 
cover $$\phi\, :\, U\,\longrightarrow\, X^o$$ such that the $G$--torsor $\phi^*Z^o$ extends to the 
normalization of $X$ in the function field $k(U)$.
\end{df}

\begin{df}\label{df:new_nori}
Let $ X^o \, \subset\, X$ be a dense open subset.
Define $$\pi^n(X^o,\, x)\,:=\,\varprojlim_{\BL(P)\cap X^o=\emptyset} \pi^N((X,P),\, x)$$
to be the inverse limit.
\end{df}

\begin{pro}\label{prohs}
As before, $X^o\, \subset\, X$ is a dense open subset.
There is a natural homomorphism $\pi^N(X^o,\, x) \,\longrightarrow\, \pi^n(X^o,\, x)$, which
is surjective.
\end{pro}
 
\begin{proof}
Let $G$ be a finite group scheme and $f\,:\,\pi^n(X^o)\,\longrightarrow\, G$ a surjection. Then 
$f$ defines a functor $\Rep(G)\,\longrightarrow\,\Rep(\pi^N(X,P))$ for \textit{some} geometric branch 
data $P$ on $X$ such that $\BL(P)\bigcap X^o
\,=\,\emptyset$. But $\Rep(\pi^N((X,\,P)))$ is same as $\Vect^f(X,\,P)$ which is same as 
$\Vect_{\Gamma}^f(Y)$, where $$f\,:\,(Y,\,O)\,\longrightarrow\,(X,\,P)$$ is an \'etale 
$\Gamma$--cover. By Nori's result a functor $\Rep(G)\,\longrightarrow\, \Vect_{\Gamma}^f(Y)$ 
defines a $G$--torsor $W$ on $Y$ which is $\Gamma$ equivariant. 
Set $Y^o\,:=\,f^{-1}(X^o)\, .$ and
let $W^o$ be the preimage of $Y^o$ in $W$. This $W^o$ is a $G$--torsor on $Y^o$ which is 
$\Gamma$--equivariant. But $Y^o\,\longrightarrow\, X^o$ is an \'etale $\Gamma$--cover, and
hence $W^o$ descents to a $G$--torsor on $X^o$. Therefore, it defines a surjection $\pi^N(X^o)
\,\longrightarrow\, G$. This construction is compatible with epimorphism of finite group schemes,
and $\pi^n(X^o)$ is the inverse limit of its finite group scheme quotients. Consequently,
this construction gives a surjection from $\pi^N(X^o)\,\longrightarrow\, \pi^n(X^o)$. 
\end{proof}
 
\begin{pro}\label{pro-characterize-pi^n}
Let $G$ be a finite group scheme, and let $Z^o\,\longrightarrow\, X^o$ be a $G$--torsor which \'etale
locally extends to $X$. Then $G$ is a quotient of $\pi^n(X^o)$. Conversely, given a surjection
$\pi^n((X,P))\,\longrightarrow\, G$, where $P$ is such that $\BL(P)\bigcap X^o\,=\,\emptyset$ and $G$ is a finite group scheme, the associated $G$--torsor
$Z^o\,\longrightarrow\, X^o$ \'etale locally extends to $X$. 
\end{pro}

\begin{proof}
Let $Y^o\,\longrightarrow\, X^o$ be a connected \'etale cover such that the pullback of the 
$G$--torsor $Z^o$ to $Y^o$ extends to the normalization of $X$ in $k(Y^o)$ (the unique normal proper model of $Y^o$ finite over $X$). By passing to the Galois closure, we may assume that 
$f\,:\,Y\,\longrightarrow\, X$ is a Galois cover; the Galois group for $f$ will be denoted by 
$\Gamma$. Let $P$ be the branch data on $X$ associated to $f$, i.e., $P\,=\,B_f$ in the 
notation of \cite{KP}. Then $f\,:\,(Y,\,O)\,\longrightarrow\, (X,\,P)$ is an \'etale 
$\Gamma$--cover. Also, the pull back of the $G$--torsor $Z^o\,\longrightarrow\, X^o$ to $Y^o$ 
and its extension to $Y$ is a $\Gamma$--equivariant $G$--torsor. Now a representation $V$ of 
$G$ induces an essentially finite $\Gamma$--equivariant bundle $\scrV$ on $Y$. The Tannaka 
subcategory generated by $\scrV$ in the Tannaka category of $\Gamma$--equivariant essentially 
finite bundles on $Y$ induces a surjection $\pi^N((X,P))\,\longrightarrow\, G$. Hence we get a 
surjection $\pi^n(X^o)\,\longrightarrow\,G$.

For the converse, first note that since $G$ is a finite group scheme, a surjection $\pi^n(X^o)$ 
factors through $\pi^N(X,P)$ for some branch data $P$ such that $$BL(P)\bigcap X^o\,=\,\emptyset\, .$$ Let 
$$f\,:\,(Y,\,O)\,\longrightarrow\, (X,\,P)$$ be an \'etale $\Gamma$--Galois cover of formal 
orbifolds. The surjection $\pi^N((X,P))\,\longrightarrow\,G$ by Tannaka formalism yields a finite 
collection $S$ of essentially finite $\Gamma$--equivariant bundle on $Y$ such that the Tannaka 
dual of the Tannaka subcategory generated by $S$ is $G$. This by an equivariant version of Nori's 
reconstruction, \cite[Section 2]{BDP}, yields a $\Gamma$--equivariant $G$--torsor on $Y$. This 
torsor restricts to a $\Gamma$--equivariant $G$--torsor on $$Y^o\,=\,f^{-1}(X^o)\, .$$ But $Y^o 
\,\longrightarrow\,X^o$ is an \'etale $\Gamma$ cover. Hence by Galois descent we get a 
$G$--torsor on $X^o$ and by construction it \'etale locally extends to $X$.
\end{proof}

Let $P$ and $Q$ be two branch data on a normal variety $X$.
We say that $P\, \ge \, Q$ 
if for all points $x\,\in\, X$ of codimension at least one and for every affine connected open 
neighborhood $U$ of $x$, $$P(x,\,U)\, \supset \, Q(x,\,U)\, .$$

\begin{pro}
Let $X$ be a smooth projective variety over $k$, and let $P\,\ge\, Q$ be two geometric branch 
data on $X$. There is a fully faithful functor $${\rm Vect}^f(X,Q)\,\longrightarrow\,{\rm Vect}^f(X,P)$$ 
that makes $\Vect^f(X,Q)$ into a Tannakian subcategory of $\Vect^f(X,P)$. In particular, this 
functor induces an epimorphism $\pi^N((X,P))\,\longrightarrow\, \pi^N((X,Q))$.
\end{pro}

\begin{proof}
This is proved in \cite[Theorem 3.7]{KP}. We note that although \cite[Theorem 3.7]{KP} is stated
for curves, its proof works, without any change, for all dimensions.
\end{proof}

\section{The kernel of projection from $\pi^N((X,P))$ to $\pi_1^{et}((X,P))$}

Let $\scrX\,=\,(X,\,P)$ be a proper formal orbifold, and let $\Vect^f(\scrX)$ be the Tannakian category of 
essentially finite vector bundles on $\scrX$. We now define a new category $\Vect^f_{et}(\scrX)$.

An object of this category is a pair $\{f:(Y,Q)\to (X,P),\, V\}$, where $f$ is \'etale and $V$ 
is an object of $\Vect^f(Y,Q)$ (i.e., $V$ is an essentially finite vector bundle on $(Y,Q)$). 
Let $\{f_i:(Y_i,Q_i)\to (X,P),\,V_i\}$, $i\,=\,1,\,2$, be two objects, and let 
$$f\,:\,(Y,\,Q)\,\longrightarrow\, (X,\,P)$$
be any \'etale morphism dominating $f_1$ and $f_2$; let $g_i\,:\,Y\,\longrightarrow\, Y_i$ be 
the morphisms through which $f$ factors. Define 
$$\Hom((f_1,V_1),\,(f_2,V_2))\,:=\,\varinjlim_{f:(Y,Q)\to (X,P)} {\rm Hom}_{{\rm 
Vect}^f(Y,Q)}(g_1^*V_1,g_2^*V_2)\, ;$$
here the limit is over all \'etale morphisms $f$ dominating $f_1$ and $f_2$. Since 
$\Vect^f(Y,\,Q)$ is an abelian category for any proper formal orbifold $(Y,\,Q)$, the category 
$\Vect^f_{et}(\scrX)$ is also abelian.

The tensor product $(f_1,\,V_1)\otimes(f_2,\,V_2)$ is defined as follows: let 
$f\,:\,(Y,\,Q)\,\longrightarrow\, (X,\,P)$ be a dominating connected component of the fiber 
product of $f_1$ and $f_2$, and let $p_1$ and $p_2$ be the natural projection morphisms from 
this fiber product. Then
$$(f_1,\,V_1)\otimes(f_2,\,V_2)\,=\,(f,\,p_1^*V_1\otimes p_2^*V_2)\, .$$
The dual of $(f_1,\,V_1)$ is $(f_1,\, V^\vee_1)$. So
$\Vect^f_{et}(\scrX)$ is a rigid tensor abelian category.

Let $x$ be a closed point of the complement $X\setminus \BL(P)$. Let $\widetilde x$ be a point 
of the universal cover $\widetilde \scrX$ of $\scrX$; this means that for every finite \'etale 
connected cover $(Y,\,Q) \,\longrightarrow\,(X,\,P)$ we choose a point in $Y$ over $x$ in a 
compatible way. The point $\widetilde x$ defines a fiber functor $\scrF_{\tilde x,\scrX}$ from 
$\Vect^f_{et}(\scrX)$ to the category of vector spaces $\Vect_k$ by sending $\{f:(Y,Q)\to 
(X,P), V\}$ to the stalk of $V$ at the image of $\widetilde x$ in $Y$. This makes 
$\Vect^f_{et}(\scrX)$ into a Tannakian category. Let corresponding proalgebraic group scheme
will be denoted by $S(X,\, P)$.

\begin{thm}\label{thm2}
Let $\scrX\,=\,(X,\,P)$ be a projective smooth formal orbifold. The dual group of the
Tannakian category 
$(\Vect^f_{et}(X,P),\, \scrF_{\widetilde{x},\scrX})$ is the kernel $$K(X,P)\,:=\,{\rm kernel}(\pi^N(\scrX,\,x)
\,\rightarrow\, \pi_1^{et}(\scrX,\,x))\, .$$
\end{thm}

\begin{proof}
Let $S(X,P)$ denote the Tannaka dual of the category $(\Vect^f_{et}(X,P),\, \scrF_{\widetilde x,\scrX})$.
Let $$f\,:\,\scrY\,\longrightarrow\,\scrX$$ be a finite connected \'etale cover with $y
\,\in\, \scrY$ being the image of $\widetilde x$. Note that there is a natural functor of Tannakian categories
\begin{equation}\label{eiy}
\scrI_{\scrY}\,:\,{\rm Vect}^f(\scrY)\,\longrightarrow\,{\rm Vect}^f_{et}(\scrX)
\end{equation}
that sends an essentially finite
vector bundle $V$ on $\scrY$ to $(f,\,V)$. This functor is a full embedding. Note that $\scrF_{\widetilde{x},\scrY}
\,:=\,\scrF_{\widetilde{x},\scrX}\circ\scrI_{\scrY}$ is a fiber functor from $\Vect^f(\scrY)$ to
the category ${\rm Veck}_k$ of $k$--vector spaces. The functor $\scrI_{\scrY}$ in \eqref{eiy}
induces a homomorphism of the duals
\begin{equation}\label{ehd}
S(X,\,P)\,\longrightarrow\, \pi^N(\scrY,\,y)\, .
\end{equation}
Also the pullback $f^*$ defines a
functor $\Vect^f(\scrX)\,\longrightarrow\,\Vect^f(\scrY)$, and we have an isomorphism of the functors $I_{\scrX}$ and
$I_{\scrY}\circ f^*$. Hence the homomorphism $S(X,P)\,\longrightarrow\,\pi^N(\scrX,\,x)$, constructed
using the homomorphisms in \eqref{ehd}, factors through $S(X,P)
\,\longrightarrow\,\pi^N(\scrY, \,y)$ for every finite \'etale cover $\scrY\,\longrightarrow\, \scrX$.
Consequently, the image of $S(X,P)$ in $\pi^N(\scrX,\,x)$ lies in $K(X,P)$.
 
Let $\{f:\scrY\to\scrX, V\}$ be an object of $\Vect^f_{et}(\scrX)$. Then $V$ embeds into 
$f^*f_*V$. Also for a vector bundle $W$ on $\scrX$ the objects $\{f:\scrY\to \scrX,\,f^*W\}$ and 
$\{id:\scrX\to \scrX,\,W\}$ are isomorphic. Hence $\{f:\scrY\to\scrX, V\}$ is a subobject of 
$\scrI_{\scrX}(V)$. So an automorphism of $\scrF_{\tilde x,\scrX}$ which restricts to identity 
automorphism on the category $\Vect^f(\scrX)$ must be identity. Hence the induced homomorphism 
$S(X,P)\,\longrightarrow\, K(X,P)$ is injective.
 
Let $\Phi$ be an automorphism of the fiber functor $F_x\,:\,\Vect^f(\scrX)\,\longrightarrow\, \Vect_k$ such that
its image in $\pi_1^{et}(\scrX,\,x)$ is trivial. So $\Phi\,\in\, \pi^N(\scrY,\,y)$ for every \'etale connected
covering $\scrY\,\longrightarrow\, \scrX$ and any point $y\,\in\, Y$ lying above $x$. Therefore,
$\Phi$ is an automorphism of the fiber functor
$$F_y\,:\,\Vect^f(\scrY)\,\longrightarrow\,\Vect_k\, .$$ Let $O\,:=\,\{f:\scrY\to \scrX,V\}$ be
an object of $\Vect^f_{et}(\scrX)$; define a map
$\widetilde \Phi$ from $\scrF_{\tilde x,\scrX}(O)$ to itself to be the map $\Phi$ from $F_y(V)$ to itself. Note that 
$\widetilde \Phi$ defines an automorphism of the fiber functor $\scrF_{\widetilde{x},\scrX}$ whose restriction to $F_x$ 
is $\Phi$. Hence the natural map $S(X,P)\,\longrightarrow\, K(X,P)$ is also surjective and so
it, being injective also, is an isomorphism.
\end{proof}

\begin{cor}\label{corl}
 Let $X$ be a projective normal variety, and let $P\,\ge\, Q$ be two geometric branch data on $X$. Then we have
the following morphism of exact sequences in which all the vertical arrows are surjective:
\[\xymatrix{
1 \ar[r] & K(X,P)\ar[r]\ar[d] & \pi^N((X,P)) \ar[r]\ar[d]& \pi_1^{et}((X,P))\ar[r]\ar[d] & 1 \\
1 \ar[r] & K(X,Q)\ar[r] & \pi^N((X,Q)) \ar[r] & \pi_1^{et}((X,Q))\ar[r] & 1
}\]
\end{cor}

\begin{proof}
The surjectivity of the third arrow follows from \cite{formal.orbifolds}. The surjectivity of the first (respectively, 
second) arrow follows from the observation that $\Vect^f(X,\,Q)$ (respectively, $\Vect^f_{et}(X,\,
Q)$) is a fully 
faithful subcategory of $\Vect^f(X,\,P)$ (respectively, $\Vect^f_{et}(X,\,P)$).
\end{proof}

\begin{cor}
 Let $X$ be a projective normal variety and $X^o$ be an open subset of $X$ such that $X\setminus X^o$ is a normal crossing divisor. Then we have the following morphism of exact sequences
in which all the vertical arrows are surjective:
\[\xymatrix{
1 \ar[r] & K^o \ar[r]\ar[d] & \pi^n(X^o) \ar[r]\ar[d]& \pi_1^{et}(X^o)\ar[r]\ar[d] & 1\\
1 \ar[r] & K\ar[r] & \pi^N(X) \ar[r]& \pi_1^{et}(X)\ar[r] & 1
}\]
\end{cor}

\begin{proof}
 This follows from Corollary \ref{corl} by taking $Q$ to be the trivial branch data and taking the inverse limit over the branch data $P$ whose branch locus lie in $X\setminus X^o$. 
\end{proof}

\begin{remark}
Let $K^t$ be the kernel of the epimorphism $\pi^N(X^o) \,\longrightarrow\,\pi_1^n(X^o)$. It is not clear to us whether $K^t$ is 
trivial. But the image of $K^t$ in $K$ under the natural homomorphism induced from $\pi^N(X^o)\,\longrightarrow\,\pi^N(X)$ is trivial, since 
$\pi^N(X^o)\,\longrightarrow\,\pi^N(X)$ factors through $\pi_1^n(X^o)$.
\end{remark}

\begin{example} Let $X\,=\,\PP^1$, $Q\,=\,O$, $P$ to be tame ramification at four points of $\PP^1$ of order 2 
(i.e., characteristic $p\ne 2$). Let $E\,\longrightarrow\, X$ be a $\Z/2\Z$--cover by an elliptic curve of 
$X\,=\,\PP^1$. Let $V$ be a non-trivial Frobenius-trivial $\Z/2\Z$--equivariant bundle on the elliptic curve. 
This can be constructed by starting with a non-trivial Frobenius-trivial bundle $L$ on $E$ (for instance take 
the bundle associated to $\mu_p$ torsor which arises from the kernel of the Frobenius morphism). Let 
$V\,=\,L\oplus g^*L$ where $g\,\in\, \Z/2\Z$ is the nontrivial element. This shows that $K(\PP^1,P)$ is 
non-trivial but $K(\PP^1,Q)$ is trivial (as $\pi^N((\PP^1,Q))\,=\,\pi^N(\PP^1)$ is trivial). Hence 
$K(X,P)\,\longrightarrow\, K(X,Q)$ is not an isomorphism. In particular, the map $K^o\,\longrightarrow\, K$ in 
the above corollary need not be an isomorphism. This also demonstrates that
$\pi^N((X,P))\,\not\cong\,\pi^N(X)\times_{\pi_1^{et}(X)}\pi_1^{et}((X,P))$ in general.
\end{example}

\end{document}